\documentclass[11pt, a4paper,twoside]{scrartcl}
\KOMAoptions{DIV=calc}
\usepackage[english]{babel}

\usepackage{color,enumitem,graphicx, caption}
\usepackage[colorlinks=true,urlcolor=blue,citecolor=red,linkcolor=blue,linktocpage,bookmarksnumbered,bookmarksopen]{hyperref}

\usepackage[normalem]{ulem}
\usepackage{amsfonts,amsmath,amssymb,latexsym,soul,amsthm, tikz-cd, braket, mathtools, graphicx,array,color,colortbl,xcolor}

\DeclarePairedDelimiter{\norm}{\lVert}{\rVert}
\usepackage[T1]{fontenc}
\usepackage[utf8]{inputenc}
\usepackage{csquotes}
\usepackage{microtype}
\usepackage{mathrsfs}

\newcommand{\dz}{{\mathrm{d}z}}

\newcommand{\dmu}{{\mathrm{d}\mu}}
\newcommand{\dnu}{{\mathrm{d}\nu}}

\usepackage{fixmath}

\newtheorem{lemma}{Lemma}[section]
\newtheorem{thm}[lemma]{Theorem}
\newtheorem{prop}[lemma]{Proposition}
\newtheorem{cor}[lemma]{Corollary}

\theoremstyle{definition}
\newtheorem{defn}[lemma]{Definition}

\numberwithin{equation}{section}

\begin{document}
	\title{Bifurcation and multiplicity results for critical problems involving the $p$-Grushin operator}
	\author{Giovanni Molica Bisci\thanks{This work has been funded by the European Union - NextGenerationEU within the framework of PNRR  Mission 4 - Component 2 - Investment 1.1 under the Italian Ministry of University and Research (MUR) program PRIN 2022 - grant number 2022BCFHN2 - "Advanced theoretical aspects in PDEs and their applications" - CUP: H53D23001960006 and Partially supported by GNAMPA 2024: "Aspetti geometrici e analitici di alcuni problemi locali e non-locali in mancanza di compattezza"} \and Paolo Malanchini \and Simone Secchi\thanks{Partially supported by GNAMPA 2024: "Aspetti geometrici e analitici di alcuni problemi locali e non-locali in mancanza di compattezza"}}
	\date{}
	\maketitle
	\setcounter{secnumdepth}{2}
	\setcounter{tocdepth}{2}
	
\begin{abstract}
\noindent \textbf{Abstract.} In this article we prove a bifurcation and multiplicity result for a critical problem involving a degenerate nonlinear operator $\Delta_\gamma^p$. We extend to a generic $p>1$ a result which was proved only when $p=2$. When $p\ne 2$, the nonlinear operator $-\Delta_\gamma^p$ has no linear eigenspaces, so our extension is nontrivial and requires an abstract critical theorem which is not based on linear subspaces. We also prove a new abstract result based on a pseudo-index related to the $\mathbb{Z}_2$-cohomological index that is applicable here.
We provide a version of the Lions’ Concentration-Compactness Principle for our operator.
\end{abstract}

\section{Introduction}
Suppose that $N \geq 3$, and that $N=m+\ell$ for some positive integers $m$ and $\ell$. Let $z= (x,y)$ be a generic point of $\mathbb{R}^N = \mathbb{R}^m\times\mathbb{R}^\ell$ and let $\gamma \geq 0$ be a real parameter. The Grushin-Baouendi operator is defined as
\begin{equation} \label{grushin}
\Delta_\gamma u (z) = \Delta_x u(z) + \lvert x \rvert^{2\gamma} \Delta_y u (z),
\end{equation}
where $\Delta_x$ and $\Delta_y$ are the Laplace operators in the variables $x$ and $y,$ respectively.

It follows from \eqref{grushin} that the Grushin operator is not uniformly elliptic in the space $\mathbb{R}^N$, since it is degenerate on the subspace $\lbrace 0\rbrace \times \mathbb{R}^\ell$. For more technical reasons, see for instance the survey \cite{egorov}, the Grushin operator belongs to the class of \emph{subelliptic} operators, which lies halfway between elliptic and hyperbolic operators.

As remarked in \cite{kogoj_lanconelli}, if $\gamma$ is a non-negative even integer, the operator $\Delta_\gamma$ falls into the class of H\"{o}rmander operators, which are defined to be a sum of squares of vector fields generating a Lie algebra of maximum rank at any point of $\mathbb{R}^N$.
If we introduce the family of vector fields
\begin{displaymath}
X_i = \frac{\partial}{\partial x_i}, ~ i=1,\dots, n,\quad
 X_{m+j} = \lvert x \rvert^\gamma \frac{\partial}{\partial y_j}, ~ j=1,\dots,\ell
\end{displaymath}
the corresponding Grushin gradient is
\begin{displaymath}
\nabla_\gamma= \left(\nabla_x, \lvert x \rvert^\gamma \nabla_y\right) = \left(\frac{\partial}{\partial x_i}, \dots, \frac{\partial}{\partial x_m}, \lvert x \rvert^\gamma \frac{\partial}{\partial y_1}, \dots, \lvert x \rvert^\gamma \frac{\partial}{\partial y_\ell}\right)
\end{displaymath}
and the Grushin operator can be written as
\begin{displaymath}
	\Delta_\gamma = \sum_{i=1}^{N} X_i^2= \nabla_\gamma \cdot \nabla_\gamma.
\end{displaymath}

In a similar manner to how the Grushin operator can be constructed from the Laplacian, it is possible to define a quasilinear Grushin-type operator based on the $p$-Laplace operator. The $p$-Laplace Grushin operator ($p$-Grushin for short) is defined by
\begin{equation} \label{p-grushin}
	\Delta_\gamma^p u = \nabla_\gamma \cdot \left(\lvert \nabla_\gamma u \rvert^{p-2}\nabla_\gamma u\right).
\end{equation}

The $p$-Grushin operator has obtained increasing attention in recent years. Recent works, such as those by Huang and Yang \cite{huang_p_laplace} and Huang, Ma and Wang \cite{huang_ma_wang}, have explored the existence, uniqueness, and regularity of solutions to the $p$-Laplace equation involving Grushin-type operators. Moreover in \cite{wei_chen_chen_yang} the authors proved a Liouville-type theorem for stable solutions of weighted $p$-Laplace-type Grushin equations.
\medskip

In this paper we are concerned with a critical problem associated to the operator \eqref{p-grushin}. More precisely, we consider the boundary value problem
\begin{equation}
	\label{problem:main}
	\begin{cases}
		 - \Delta_\gamma^p u = \lambda \lvert u \rvert^{p-2} u + \lvert u \rvert^{p^*_\gamma -2}u \quad &\hbox{in $\Omega$},\\
		 u= 0 \quad &\hbox{in $\partial\Omega$},
	\end{cases}
\end{equation}
where $\Omega\subset\mathbb{R}^N$ is bounded and $p^*_\gamma$ is the critical Sobolev exponent for the $p$\nobreakdash-Grushin operator defined in \eqref{pstar} below.

In the seminal paper \cite{brezis_nirenberg} Brezis and Nirenberg considered the problem with Sobolev critical growth
 \begin{equation}
 	\label{problem:brezis_niremberg}
 	\begin{cases}
 		-\Delta u = \lambda u + \lvert u \rvert^{2^*-2} u \quad &\hbox{in $\Omega$},\\
 		u=0 \quad &\hbox{in $\partial\Omega$},
 	\end{cases}
 \end{equation}
 on a bounded domain $\Omega\subset\mathbb{R}^N$.
 Here $N \geq 3$, $2^*=2N/(N-2)$ is the critical Sobolev exponent, and $\lambda >0$ is a positive parameter. By a careful comparison argument, the authors proved that for $N \geq 4$,
 \begin{displaymath}
 	\inf_{\substack{u \in H_0^1(\Omega) \\ \int_\Omega \vert u \vert^{2^*} =1}} \left\lbrace \int_\Omega \vert \nabla u \vert^2 - \lambda \int_\Omega \vert u \vert^2 \right\rbrace < \inf_{\substack{u\in H_0^1(\Omega) \\ \int_\Omega \vert u \vert^{2^*}=1}} \int_\Omega \vert \nabla u \vert^2 = S,
 \end{displaymath}
 where $S$ is the best constant for the continuous Sobolev embedding $H_0^1(\Omega) \subset L^{2^*}(\Omega)$. Calling $\lambda_1$ the first eigenvalue of the operator~$-\Delta$ defined on $H_0^1(\Omega) \subset L^2(\Omega)$, the existence of a positive solution to \eqref{problem:brezis_niremberg} follows for every $\lambda \in (0,\lambda_1)$ if $N \geq 4$. The case $N=3$ is harder, and the authors proved the existence of a positive solution when $\Omega$ is a ball and $\lambda$ is sufficiently close to $\lambda_1$.

 The paper \cite{brezis_nirenberg} marked the beginning of an endless stream of efforts to extend the seminal results. In 1984, Cerami, Fortunato and Struwe proved in \cite[Theorem 1.1]{cerami1984bifurcation} the following multiplicity result for problem \eqref{problem:brezis_niremberg}. Let $0<\lambda_1<\lambda_2\le \lambda_3\le\dots \to+\infty$ be the Dirichlet eigenvalues of $-\Delta$ in $\Omega$. If $\lambda_k\le\lambda<\lambda_{k+1}$ and
 \begin{displaymath}
 \lambda>\lambda_{k+1} - \frac{S}{\lvert \Omega \rvert^{2/N}},
 \end{displaymath}
 then \eqref{problem:brezis_niremberg} has $m$ distinct pairs of nontrivial solutions $\pm u_j^\lambda$, where $m$ is the multiplicity of the eigenvalue $\lambda_{k+1}$.

 Similar multiplicity and bifurcation results have been proved for several classes of variational operators. For instance, the authors considered in \cite{fiscella2018multiplicity} a rather general family of fractional (i.e. nonlocal) operators which contains the usual fractional Laplace operator in $\mathbb{R}^N$.

Subsequently,  a multiplicity result has been proved in \cite{perera_multiplicity} for the corresponding problem for the $p$-Laplacian
\begin{displaymath}
	\begin{cases}
		- \Delta_p u = \lambda \lvert u \rvert^{p-2} u + \lvert u \rvert^{p^* -2}u \quad &\hbox{in $\Omega$},\\
		u= 0 \quad &\hbox{in $\partial\Omega$},
	\end{cases}
\end{displaymath}
where $\Delta_p u = \nabla\cdot \left(\lvert \nabla u \rvert^{p-2}\nabla u\right)$ is the $p$-Laplace operator and $p^* = Np/(N-p)$.
\medskip

Finally, the authors in \cite{perera_multiplicity_fractional} extended the same multiplicity result for a critical problem involving the fractional $p$-Laplace operator.
\medskip

The case $p=2$ of \eqref{problem:main} was studied in \cite{loiudice_brezis}, where the authors extended the Brezis-Nirenberg result to the critical problem with the Grushin operator. A multiplicity result for the case $p=2$ was proved in \cite{bisci2024note}.
\medskip

In this paper we extend the above bifurcation and multiplicity results to the problem~\eqref{problem:main}. This extension requires some care. Indeed, the linking argument bases on the eigenspace of $-\Delta_\gamma$ in \cite{bisci2024note} does not work when $p\ne 2$ since the nonlinear operator $-\Delta_\gamma^p$ does not have linear eigenspaces. We will use a different sequence of eigenvalues that is based on the $\mathbb{Z}_2$-cohomological index of Fadell and Rabinowitz \cite{fadell_rabinowitz}.
\medskip

In what follows, $S$ stands for the optimal constant for the Sobolev embedding to be defined in \eqref{eq:sobconst}, and $\lvert\Omega\rvert$ denotes the Lebesgue measure of a bounded domain $\Omega\subset\mathbb{R}^N$.

If we denote by $\lbrace \lambda_k\rbrace$ the sequence of eigenvalues of the $p$\nobreakdash-Grushin operator introduced in Section \ref{sec:eigenvalues}, we can state the main result of this paper.

\begin{thm}
	\label{thm:main}
The following facts hold:
\begin{enumerate}
	\item[$(1)$] If
	\begin{displaymath}
	\lambda_1 - \frac{S}{\lvert \Omega \rvert^{p/N_\gamma}}<\lambda <\lambda_1,
	\end{displaymath}
	then problem \eqref{problem:main} has a pair of nontrivial solutions $\pm u^\lambda$ such that $u^\lambda\to 0$ as $\lambda\nearrow \lambda_1$.
	\item[$(2)$] If $\lambda_k\le \lambda <\lambda_{k+1} = \dots = \lambda_{k+m}<\lambda_{k+m+1}$ for some $k,m\in\mathbb{N}$ and
	\begin{equation}
		\label{eq:lambdak}
		\lambda > \lambda_{k+1} - \frac{S}{\lvert \Omega \rvert^{p/N_\gamma}},
	\end{equation}
	then problem \eqref{problem:main} has $m$ distinct pairs of nontrivial solutions $\pm u^\lambda_j, \, j=1,\dots, m,$ such that $u^\lambda_j\to 0$ as $\lambda \nearrow \lambda_{k+1}$.
	\end{enumerate}
\end{thm}

In particular, we have the following existence result.
\begin{cor}
	Problem \eqref{problem:main} has a nontrivial solution for all
	\begin{equation*}
		\lambda\in \bigcup_{k=1}^\infty \left(\lambda_k - S/\lvert \Omega \rvert^{p/N_\gamma}, \lambda_{k} \right).
	\end{equation*}
\end{cor}

The paper is organized as follows. In Section \ref{sec:functional_setting}, we introduce the functional space $\mathring{W}^{1,p}_\gamma(\Omega)$.
	In Section \ref{sec:concentration_compactness}, we prove a Concentration-Compactness Principle for the $p$-Grushin operator.
	Section \ref{sec:eigenvalues} focuses on the eigenvalues of the operator $-\Delta_p$.
Section \ref{sec:abstract} recalls an abstract multiplicity result, which plays a key role in the proof of Theorem \ref{thm:main}, presented in Section \ref{sec:proof}.

\section{Functional setting}
\label{sec:functional_setting}
Fix a bounded domain $\Omega$ and a number $1<p<\infty$. The Sobolev space $\mathring{W}^{1,p}_\gamma(\Omega)$ is defined as the completion of $C_c^1(\Omega)$ with respect to the norm
\begin{displaymath}
\norm{u}_{\gamma,p} = \left(\int_\Omega \lvert \nabla_\gamma u \rvert^{p}\,\dz\right)^{1/p}.
\end{displaymath}
Specially, the space $\mathring{W}^{1,2}_\gamma(\Omega) = \mathscr{H}(\Omega)$ is a Hilbert space endowed with the inner product
\begin{displaymath}
	\braket{u,v}_\gamma = \int_{\Omega}^{} \nabla_\gamma u \nabla_\gamma v\, \, \dz.
\end{displaymath}

The following embedding result was proved in \cite[Proposition 3.2 and Theorem 3.3]{kogoj_lanconelli}. See also \cite[Corollary 2.11]{huang_p_laplace}.
\begin{prop}
	Let $\Omega\subset\mathbb{R}^N$ be a bounded open set. Then the embedding
	\begin{displaymath}
		\mathring{W}^{1,p}_\gamma(\Omega) \hookrightarrow L^q(\Omega)
	\end{displaymath}
	is compact for every $q\in [1,p^*_\gamma)$, where
	\begin{equation}
		\label{pstar}
		p^*_\gamma = \frac{p N_\gamma}{N_\gamma - p}
	\end{equation}
	and
	\begin{gather*}
		N_\gamma = m + (1+\gamma) \ell
	\end{gather*}
	is the homogeneous dimension of $\mathbb{R}^N$ associated to the decomposition $\mathbb{R}^N = \mathbb{R}^m\times \mathbb{R}^\ell$.
\end{prop}

We may thus define the best constant of the Sobolev embedding $\mathring{W}^{1,p}_\gamma(\Omega)\hookrightarrow L^{p^*_\gamma}(\Omega)$
\begin{equation}
	\label{eq:sobconst}
	S = \inf_{\substack{ u \in \mathring{W}^{1,p}_\gamma(\Omega) \\ u \neq 0}} \frac{\int_\Omega \left\vert \nabla_\gamma u \right\vert^p \, \dz}{\left( \int_\Omega \left\vert u \right\vert^{p_\gamma^*} \, \dz \right)^{p/p_\gamma^*}}.
\end{equation}



\section{Concentration-Compactness}
\label{sec:concentration_compactness}
In this section we adapt the arguments from \cite{lions,lions2} and \cite[Theorem 4.8]{struwe} to establish a concentration-compactness result for the $p$\nobreakdash-Grushin operator.

Let $\mathcal{M}(\mathbb{R}^N)$ be the space of all finite signed Radon measures. We recall the definition of \emph{tight convergence of measures}.
\begin{defn}
	A sequence of measures $\{\mu_n\}\subset\mathcal{M}(\mathbb{R}^N)$ converges tightly to a measure $\mu\in\mathcal{M}(\mathbb{R}^N)$, that is  $\mu_n\buildrel\ast\over\rightharpoonup\mu$, if for every $\varphi\in C_b(\mathbb{R}^N)$
	\begin{displaymath}
		\int_{\mathbb{R}^N} \varphi \,\mathrm{d}\mu_n \to 	\int_{\mathbb{R}^N} \varphi\,\dmu \quad \hbox{as $n\to\infty$},
	\end{displaymath}
	where $C_b(\mathbb{R}^N)$ is the space of the bounded, continuous functions on $\mathbb{R}^N$.
\end{defn}
We will also need the following lemma, see \cite[Lemma 1.2 and Remark 1.5]{lions}.
\begin{lemma}
	\label{lemma:lions}
	Let $\mu,\nu$ two bounded non-negative measures on $\Omega$ satisfying for some constant $C_0\ge 0$
	\begin{displaymath}
		\left(\int_{\Omega} \lvert \varphi \rvert^q\,\dnu\right)^{1/q}	\le C_0 \left(\int_{\Omega}\lvert \varphi\rvert^p\,\dmu \right)^{1/p} \quad \forall \varphi\in C^\infty_c(\Omega)
	\end{displaymath}
	where $1\le p\le q\le +\infty$. Then, there exist an at most countable set $J$, families $\lbrace x_j\rbrace_{j\in J}$ of distinct points in $\overline{\Omega}$ and positive numbers $\lbrace \nu_j\rbrace_{j\in J}$ such that
	\begin{displaymath}
		\nu= \sum_{j\in J} \nu_j \delta_{x_j}, \quad \mu \ge C_0^{-p} \sum_{j\in J} \nu_j^{p/q}\delta_{x_j}.
	\end{displaymath}
\end{lemma}

Now we are ready to prove the Concentration-Compactness Principle.
\begin{thm}
	\label{thm:concentration}
	Let $\Omega$ be a bounded subset in $\mathbb{R}^N$ and $\lbrace u_n\rbrace$ be a bounded sequence in $\mathring{W}^{1,p}_\gamma(\Omega)$. Then we have:
	\begin{enumerate}
		\item[$(1)$] up to a subsequence, there exist $u\in \mathring{W}^{1,p}_\gamma(\Omega)$, two bounded non-negative measures $\mu$ and $\nu$, an at most countable set $J$, a family $\lbrace x_j\rbrace_{j\in J}$ of distinct points in $\overline{\Omega}$ and a family $\lbrace \nu_j \rbrace_{j\in J}$ of positive numbers such that
		\begin{align}
			&\notag u_n \rightharpoonup u~ \text{in}~ \mathring{W}^{1,p}_\gamma(\Omega),\\
			&\label{eq:cc1} \mu_n:= \lvert \nabla_\gamma u_n\rvert^p\dz \buildrel\ast\over\rightharpoonup \mu, \quad \nu_n:= \lvert u_n\rvert^{p^*_\gamma}\dz \buildrel\ast\over\rightharpoonup \nu,\\
			& \label{eq:cc2} \nu = \lvert u \rvert^{p^*_\gamma}\dz + \sum_{j\in J} \nu_j \delta_{x_j},
		\end{align}
		where $\delta_x$ is the Dirac measure concentrated at $x\in\mathbb{R}^N$.
		\item[$(2)$] In addition we have
		\begin{equation}
			\label{eq:cc3} \mu\ge \lvert \nabla_\gamma u\rvert^p\dz + \sum_{j\in J}\mu_j \delta_{x_j}
		\end{equation}
		for some family of positive numbers $\lbrace \mu_j\rbrace_{j\in J}$ satisfying
		\begin{equation}
			\label{eq:cc4} \mu_j \ge S \nu_j^{\frac{p}{p^*_\gamma}}\quad\hbox{for all $j\in J$},
		\end{equation}
		where $S$ is defined in \eqref{eq:sobconst}. In particular,
		\begin{displaymath}
			\sum_{j\in J} (\nu_j)^{\frac{p}{p^*_\gamma}}<\infty.
		\end{displaymath}
	\end{enumerate}
\end{thm}

\begin{proof}
	Since $\Omega$ is bounded, the measures $\lvert \nabla_\gamma u_n\rvert^p\dz$ and $\lvert u_n\rvert^{p^*_\gamma}\dz$ are uniformly tight in $n.$ Therefore, by Prohorov's theorem \cite[Theorem 1.208]{fonseca_modern_methods} there exist two non-negative bounded measures $\mu, \nu$ in $\Omega$
such that \eqref{eq:cc1} holds.
	\medskip
	
	Let $\varphi\in C^\infty_c(\Omega),$ by Sobolev inequalities we have
	\begin{displaymath}
		\int_{\Omega} \lvert\varphi u_n\rvert^{p^*_\gamma}\,\dz \le S^{-p^*_\gamma/p} \left( \int_{\Omega}\lvert \nabla_\gamma(\varphi u_n)\rvert^p\,\dz\right)^{1/p}.
	\end{displaymath}
	As $n \to \infty$, the left-hand side converges to
	$
	\int_{\Omega} |\varphi|^{p^*_\gamma} \, \dnu,
	$
	while the right-hand side approaches
	$
	S^{-p^*_\gamma/p} \left( \int_{\Omega} |\varphi|^p \, \dmu \right)^{1/p},
	$
	since all the remaining lower-order terms in the expansion of $\lvert\nabla_\gamma(\varphi u_n)\rvert^p$ converge to zero in $ L^p(\Omega)$ as $n \to \infty$ because of the compactness of the Sobolev embeddings.
	
	That is holds
	\begin{equation}
		\label{eq:cc5}
		\left(\int_{\Omega}\lvert \varphi\rvert^{p^*_\gamma}\,\dnu\right)^{1/p^*_\gamma} \le S^{-1/p}\left(\int_{\Omega} \lvert \varphi\rvert^p\,\dmu\right)^{1/p}
	\end{equation}
	for all $\varphi\in C^\infty_c(\Omega)$.
	So Theorem \ref{thm:concentration} is proved in the case $u\equiv 0$ applying Lemma \ref{lemma:lions}.
	
	\medskip
	We now turn to the general case $u \rightharpoonup u$.
	Let $v_n = u_n - u\in \mathring{W}^{1,p}_\gamma(\Omega)$. Then $v_n \rightharpoonup 0$ in $ \mathring{W}^{1,p}_\gamma(\Omega)$ and by the Brezis-Lieb Lemma (see \cite[Theorem 1]{brezis_lieb}),
	\begin{equation}
		\label{eq:cc6}
		\int_\Omega \lvert u_n\rvert^{p^*_\gamma} \,\dz - \int_{\Omega} \lvert u_n - u\rvert^{p^*_\gamma} \,\dz \to \int_{\Omega} \lvert u\rvert^{p^*_\gamma}\, \dz
	\end{equation}
	as $n\to\infty$.
	Moreover, by \eqref{eq:cc6} we have
	\begin{displaymath}
		\omega_n := \nu_n -\lvert u\rvert^{p^*_\gamma}\dz = \left(\lvert u_n\vert^{p^*_\gamma} - \lvert u \rvert^{p^*_\gamma} \right)\,\dz = \lvert u_n - u \rvert^{p^*_\gamma}\dz + o(1) = \lvert v_n \rvert^{p^*_\gamma}\dz + o(1).
	\end{displaymath}
	Also let $\lambda_n = \lvert \nabla_\gamma v_n\rvert^2\dz$. We may assume that $\lambda_n \buildrel\ast\over\rightharpoonup \lambda $, while $\omega_n \buildrel\ast\over\rightharpoonup \omega = \nu - \lvert u \rvert ^{p^*_\gamma}\dz$, where $\lambda,\omega$ are positive measures.
	
	Consider $\varphi\in C^\infty_c(\Omega)$. Then, arguing as before,
	\begin{align*}
		\int_{\Omega} \lvert \varphi \rvert^{p^*_\gamma}\,\mathrm{d}\omega &= \lim_{n	\to\infty} \int_{\Omega} \lvert \varphi\rvert^{p^*_\gamma} \mathrm{d}\omega_n = \lim_{n\to\infty} \int_{\Omega}\lvert v_n\varphi\rvert^{p^*_\gamma} \,\dz \\
		&\le S^{-p^*_\gamma/p}\liminf_{n\to\infty}\left( \int_{\Omega}\lvert \nabla_\gamma(v_n\varphi)\lvert^p\,\dz\right)^{p^*_\gamma/p} \\
		&= S^{-p^*_\gamma/p} \liminf_{n\to\infty} \left( \int_{\Omega} \lvert\varphi\rvert^p \lvert \nabla_\gamma v_n\rvert^{p}\,\dz\right)^{p^*_\gamma/p} = S^{-p^*_\gamma/p}\left(\int_{\Omega} \lvert \varphi\rvert^p\,\mathrm{d}\lambda\right)^{p^*_\gamma/p},
	\end{align*}
	That is, there holds
	\begin{displaymath}
		\left(\int_{\Omega}\lvert \varphi\rvert^{p^*_\gamma}\,\mathrm{d}\omega\right)^{1/p^*_\gamma} \le S^{-1/p}\left(\int_{\Omega} \lvert \varphi\rvert^p\,\mathrm{d}\lambda\right)^{1/p}
	\end{displaymath}
	for all $\varphi\in C^\infty_c(\Omega)$. Now \eqref{eq:cc2} holds true by Lemma \ref{lemma:lions}.
	\medskip
	
	To get \eqref{eq:cc3} we first claim that $\mu \ge \lvert \nabla_\gamma u \rvert^p\,\dz$. Indeed, for any $\varphi\in C^\infty_c(\Omega),\,\varphi\ge 0$ the functional
	\begin{displaymath}
		v\mapsto \int_{\Omega} \lvert\nabla_\gamma v\rvert^p\varphi\,\dz
	\end{displaymath}
	is convex and continuous, therefore $u_n\rightharpoonup u$ in $\mathring{W}^{1,p}_\gamma(\Omega)$ implies
	\begin{displaymath}
		\int_{\Omega}\varphi\,\mathrm{d}\mu = \lim_{n\to\infty}\int_{\Omega}\lvert \nabla_\gamma u_n\rvert^p\varphi\,\dz \ge \int_{\Omega}\lvert\nabla_\gamma u\rvert^p\varphi\,\dz, \quad\hbox{for any $\varphi\in C^\infty_c(\Omega),\,\varphi \ge 0$.}
	\end{displaymath}
	
Let $\varphi\in C^\infty_c(\Omega)$ such that $0\le\varphi\le 1,\, \operatorname{supp}\varphi = B(0,1)$ and $\varphi(0) = 1$. Given $\varepsilon>0$ we apply the inequality \eqref{eq:cc5} with $\varphi\left(\frac{x-x_j}{\varepsilon}\right)$ where $j$ is fixed in $J$.
	We obtain
	\begin{displaymath}
		\nu_j^{1/p^*_\gamma} S^{1/p} \le \mu\left(B(x_j,\varepsilon)\right)^{1/p}.
	\end{displaymath}
	This implies that $\mu\left(\lbrace x_j\rbrace\right)>0$ and
	\begin{displaymath}
		\mu\ge \nu_j^{p/p^*_\gamma}S\delta_{x_j}
	\end{displaymath}
	and thus
	\begin{displaymath}
		\mu \ge \sum_{j\in J} \nu_j^{p/p^*_\gamma}S\delta_{x_j}.
	\end{displaymath}
	Since $\lbrace \lvert \nabla_\gamma u\rvert^p\,\dz\rbrace \cup \lbrace \delta_{x_j}\mid j\in J \rbrace$ is a set consisting of pairwise mutually singular measures and the latter estimate holds,  \eqref{eq:cc3} and \eqref{eq:cc4} follow.
\end{proof}

\section{Eigenvalues of the $p$-Grushin operator}
\label{sec:eigenvalues}
The Dirichlet spectrum of $-\Delta_\gamma^p$ in $\Omega$ consists of those $\lambda\in\mathbb{R}$ for which the problem
\begin{equation}
	\label{prob:eigenvalues}
	\begin{cases}
		-\Delta_\gamma^p u = \lambda \lvert u \rvert^{p-2}u \quad &\hbox{in $\Omega$},\\
		u=0 \quad & \hbox{on $\partial\Omega$},
	\end{cases}
\end{equation}
has a nontrivial weak solution $u \in \mathring{W}^{1,p}_\gamma(\Omega)$.
\medskip

The spectrum of the $2$-Grushin operator was studied in \cite[Theorem 1]{xu_chen_regan}. It consists of a sequence $\lbrace \lambda_k\rbrace$ such that
\begin{displaymath}
0<\lambda_1<\lambda_2\le \dots \le \lambda_k\le \lambda_{k+1}\le\dots
\end{displaymath}
and
\begin{displaymath}
\lambda_k\to +\infty \quad\hbox{as $k\to+\infty$}.
\end{displaymath}

A complete description of the spectrum remains unknown even for the classical $p$-Laplacian case. It is well known that the first eigenvalue $\lambda_1$ is positive, simple, and has an associate positive eigenfunction $\varphi_1$, see \cite{lindqvist,lindqvist_addendum}. Increasing and unbounded sequences of eigenvalues can be defined using various minimax schemes, but a complete list of the eigenvalues of $-\Delta_p$ remains unavailable.
\medskip

Let us denote by $W_\gamma^{-1,p'}(\Omega) =\left(\mathring{W}^{1,p}_\gamma(\Omega)\right)^*$ the dual space of $\mathring{W}^{1,p}_\gamma(\Omega)$. Introducing the operator $A_p\colon \mathring{W}^{1,p}_\gamma(\Omega)\to W_\gamma^{-1,p'}(\Omega) $  by
 \begin{displaymath}
 \langle A_p(u),v\rangle = \int_{\Omega} \lvert \nabla_\gamma u \rvert^{p-2}\nabla_\gamma u \nabla_\gamma v \,\dz,
 \end{displaymath}
a weak solution of \eqref{prob:eigenvalues} can be characterized as a function $u\in \mathring{W}^{1,p}_\gamma(\Omega)$ such that
\begin{displaymath}
\langle A_p(u),v\rangle = \lambda \int_{\Omega} \lvert u \rvert^{p-2} u v \,\dz
\end{displaymath}
for all $v\in \mathring{W}^{1,p}_\gamma(\Omega).$

We collect here some remarkable properties of the nonlinear operator $A_p \in C\left(\mathring{W}^{1,p}_\gamma(\Omega), W_\gamma^{-1,p'}(\Omega)\right)$.
\begin{enumerate}[label=(A$_\arabic*$)]
	\item  $A_p$ is $(p-1)$-homogeneous\footnote{That is $A_p(\alpha u) = \alpha^{p-1} A_p (u)$ for all $u\in \mathring{W}^{1,p}_\gamma(\Omega)$, $\alpha\ge 0$.} and odd;
	\item $A_p$ is uniformly positive: $\braket{A_p(u),u} = \norm{u}_{\gamma,p}^p$ for each $u \in \mathring{W}^{1,p}_\gamma(\Omega)$;
	\item $A_p$ is a potential operator: in fact it is the Fr\'{e}chet derivative of the functional $u\mapsto \frac{\norm{u}_{\gamma,p}^p}{p}$ in $\mathring{W}^{1,p}_\gamma(\Omega)$;
	\item $A_p$ is of type $(S)$: every sequence $\lbrace u_j\rbrace$ in $\mathring{W}^{1,p}_\gamma(\Omega)$ such that
	\begin{displaymath}
	u_j \rightharpoonup u, \quad \braket{A_p(u_j), u_j -u }\to 0
	\end{displaymath}
	has a subsequence which converges strongly to $u$ in $\mathring{W}^{1,p}_\gamma(\Omega)$.
\end{enumerate}

In particular, the compactness property (A$_4$) follows from \cite[Proposition 1.3]{perera_p_laplacian} and the fact that, by H\"{o}lder's inequality and the definition of the operator $A_p$,
\begin{displaymath}
\langle A_p(u),v\rangle \le \norm{u}_{\gamma,p}^{p-1}\norm{v}_{\gamma,p},\quad \braket{A_p(u),u} = \norm{u}_{\gamma,p}^p ~\hbox{for every $u,v\in \mathring{W}^{1,p}_\gamma(\Omega)$}.
\end{displaymath}
Now we define a non-decreasing sequence $\lbrace \lambda_k\rbrace$ of eigenvalues of $-\Delta_\gamma^p$ by means of the cohomological index. This type of construction was introduced for the $p$-Laplacian by Perera \cite{perera} (see also \cite{perera_szulkin}), and it is slightly different from the traditional one, based on the Krasnoselskii genus.

\bigskip

We recall that the $\mathbb{Z}_2$-cohomological index of Fadell and Rabinowitz \cite{fadell_rabinowitz} is defined as follows.

Let $W$ be a Banach space and let $\mathcal{A}$ denote the class of  those subsets $A$ of $W\setminus\lbrace 0\rbrace$ which are symmetric in the sense that $-A=A$. For $A\in\mathcal{A}$, let $\overline{A}= A/\mathbb{Z}_2$ be the quotient space of $A$ with each $u$ and $-u$ identified, let $f\colon \overline{A}\to \mathbb{R}P^\infty$ be the classifying map of $\overline{A}$, and let $f^*\colon H^*(\mathbb{R}P^\infty) \to H^*(\overline{A})$ be the induced homomorphism of the Alexander-Spanier cohomology rings. The cohomological index of $A$ is defined by
\begin{displaymath}
i(A) =
\begin{cases}
	\operatorname{sup} \lbrace m\ge 1\mid f^*(\omega^{n-})\ne 0\rbrace, \quad& A\ne\emptyset,\\
	0,\quad & A = \emptyset,
\end{cases}
\end{displaymath}
where $\omega\in H^1(\mathbb{R}P^\infty)$ is the generator of the polynomial ring $H^*(\mathbb{R}P^\infty)= \mathbb{Z}_2[\omega]$.
The following proposition summarizes the basic properties of the cohomological index, see \cite[Theorem 5.1]{fadell_rabinowitz}.
\begin{prop}
	\label{prop:index}
	The index $i\colon\mathcal{A}\to\mathbb{N}\cup \lbrace 0,\infty\rbrace$ satisfies the following properties
	\begin{enumerate}[label=$($i$_\arabic*)$]
		\item Definiteness: $i(A) = 0$ if and only if $A=\emptyset$$;$
		\item Monotonicity: if there is an odd continuous map from $A$ to $B$, then $i(A)\le i(B)$. Thus, equality holds when the map is an odd homeomorphism$;$
		\item Dimension: $i(A)\le \operatorname{dim} W;$
		\item Continuity: if $A$ is closed, then there is a closed neighborhood $N\in\mathcal{A}$ of $A$ such that $i(N) = i(A)$. When $A$ is compact, $N$ may be chosen to be a $\delta$-neighbourhood $N_\delta (A) = \lbrace u\in W\mid \operatorname{dist}(u,A)\le\delta\rbrace;$
		\item Subadditivity: if $A$ and $B$ are closed, then $i(A\cup B)\le i(A) + i(B);$
		\item Stability: if $SA$ is the suspension of $A\ne \emptyset$, that is the quotient space of $A\times [-1,1]$ with $A\times \lbrace 1\rbrace$ and $A\times \lbrace -1\rbrace$ collapsed to different points, then $i(SA) = i(A) +1;$
		\item Piercing property: if $A, A_0$ and $A_1$ are closed, and $\varphi\colon A\times [0,1]\to A_0\cup A_1$ is a continuous map such that $\varphi(-u,t) = -\varphi(u,t)$ for all $(u,t)\in A\times [0,1]$, $\varphi(A\times [0,1])$ is closed, $\varphi(A\times \lbrace 0\rbrace) \subset A_0$ and $\varphi(A\times \lbrace 1\rbrace)\subset A_1,$ then $i(\varphi(A\times [0,1])\cap A_0\cap A_1)\ge i(A);$
		\item Neighborhood of zero: if $U$ is a bounded closed symmetric neighborhood of $0$, then $i(\partial U) = \operatorname{dim}W.$
	\end{enumerate}
\end{prop}

We define a $C^1$-Finsler manifold $\mathcal{M}$ by setting
\begin{displaymath}
\mathcal{M} = \lbrace u\in \mathring{W}^{1,p}_\gamma(\Omega)\mid \norm{u}_{\gamma,p} = 1\rbrace.
\end{displaymath}
For all $k\in\mathbb{N}$, we denote by $\mathcal{F}_k$ the family of all closed, symmetric subsets $M$ of $\mathcal{M}$ such that $i(M)\ge k$, and set
\begin{equation}
	\label{eq:eigenvalues}
	\lambda_k= \inf_{M\in\mathcal{F}_k}\sup_{u\in M}\Psi(u),
\end{equation}
where
\begin{displaymath}
\Psi(u) = \frac{1}{\norm{u}_p^p}, \quad u\in\mathcal{M}\setminus\lbrace 0\rbrace.
\end{displaymath}
Given $a\in\mathbb{R}$ we use the standard notation for the sublevels and superlevels of $\Psi$
\begin{displaymath}
\Psi_a = \lbrace u\in \mathring{W}^{1,p}_\gamma(\Omega)\mid \Psi(u)\ge a\rbrace, \quad \Psi^a = \lbrace u\in \mathring{W}^{1,p}_\gamma(\Omega)\mid \Phi(u)\le a\rbrace.
\end{displaymath}

A straightforward application of \cite[Theorem 4.6]{perera_p_laplacian} to the operator $A_p$ yields the following spectral theory.
\begin{prop}
	The sequence $\lbrace \lambda_{k}\rbrace$ defined in \eqref{eq:eigenvalues} is a nondecreasing sequence of eigenvalues of $-\Delta_\gamma^p$. Moreover
	\begin{enumerate}
		\item[$(1)$] the smallest eigenvalue, called the \emph{first eigenvalue}, is
		\begin{displaymath}
		\lambda_1 = \min_{u\ne 0} \frac{\norm{u}_{\gamma,p}^p}{\norm{u}_p^p}>0;
		\end{displaymath}
		\item[$(2)$] we have $i(\mathcal{M}\setminus \Psi_{\lambda_k}) < k\le i(\Psi^{\lambda_k})$.
		If $\lambda_k < \lambda < \lambda_{k+1},$ then
		\begin{equation}
			\label{eq:index}
			i(\Psi^{\lambda_k}) = i(\mathcal{M}\setminus \Psi_\lambda) = i(\Psi^\lambda) = i(\mathcal{M}\setminus\Psi_{\lambda_{k+1}})=k;
		\end{equation}
 		\item[$(3)$] $\lambda_{k}\to+\infty$ as $k\to+\infty$.
	\end{enumerate}
\end{prop}

\section{An abstract critical point theorem}
\label{sec:abstract}
Consider an even functional $\Phi$ of class $C^1$ on a Banach space $W$, let $\mathcal{A}^*$ denote the class of symmetric subsets of $W$, let $r>0$ and $S_r = \lbrace u\in W\mid \norm{u} = r\rbrace,$ let $0<b\le+\infty.$ Let $\Gamma$ denote the group of odd homeomorphisms of $W$ that are the identity outside $\Phi^{-1}(0,b)$. The pseudo-index of $M\in\mathcal{A}^*$ related to $S_r$ and $\Gamma$ is defined by
\begin{displaymath}
i^*(M) = \min_{\gamma\in\Gamma} i (\gamma(M)\cap S_r).
\end{displaymath}
To get our result, we will apply the following critical point theorem, whose proof can be found in \cite[Theorem 2.2]{perera_multiplicity}.
\begin{thm}
	\label{thm:critical_point}
	Let $A_0$, $B_0$ be symmetric subsets of $S_1$ such that $A_0$ is compact, $B_0$ is closed, and
	\begin{displaymath}
	i(A_0)\ge k+m, 	\quad i(S_1\setminus B_0)\le k
	\end{displaymath}
	for some integers $k\ge 0$ and $m\ge 1$. Assume that there exists $R>r$ such that
	\begin{displaymath}
	\operatorname{sup}\Phi(A)\le 0 < \operatorname{inf}\Phi(B), \quad \operatorname{sup}\Phi(X)<b,
	\end{displaymath}
	where
	\begin{align*}
		A &= \lbrace Ru\mid u\in A_0\rbrace, \\
		B &= \lbrace ru\mid u\in B_0\rbrace, \\
		X &= \lbrace tu\mid u\in A, 0\le t\le 1\rbrace.
	\end{align*}
	For $j=k+1,\dots, k+m,$ we set
	\begin{displaymath}
	\mathcal{A}_j^* = \lbrace M\in \mathcal{A}^*\mid \hbox{$M$ is compact and $i^*(M)\ge j$}\rbrace
	\end{displaymath}
	and
	\begin{displaymath}
	c_j^* = \inf_{M\in\mathcal{A}_j^*} \max_{u\in M} \Phi(u).
	\end{displaymath}
	Then
	\begin{displaymath}
	\inf \Phi(B) \le c_{k+1}^*\le \dots \le c_{k+m}^* \le \sup \Phi(X),
	\end{displaymath}
	in particular, $0<c_j^*<b.$ If, in addition, $\Phi$ satisfies the Palais-Smale condition for all levels $c\in (0,b)$, \footnote{This means that every sequence $\{u_j\}$ in $W$ such that $\Phi(u_j) \to c$ and $\Phi'(u_j) \to 0$ as $j \to +\infty$ admits a subsequence which converges strongly in $W$.} then each $c_j^*$ is a critical value of $\Phi$ and there are $m$ distinct pairs of associated critical points.
\end{thm}

\section{Proof of Theorem \ref{thm:main}}
\label{sec:proof}
In this section we prove Theorem \ref{thm:main}. Solutions of problem \eqref{problem:main} coincide with critical points of the $C^1$-functional $I_\gamma \colon \mathring{W}^{1,p}_\gamma(\Omega) \to \mathbb{R}$ defined by
\begin{displaymath}
	I_\gamma(u) = \frac{1}{p} \int_{\Omega}^{}(\lvert \nabla_\gamma u \rvert^p - \lambda \lvert u \rvert^p) \, \dz -\frac{1}{p^*_\gamma}\int_{\Omega}^{} \lvert u \rvert^{p^*_\gamma} \, \dz.
\end{displaymath}
In order to apply Theorem \ref{thm:critical_point} to the functional $I_\lambda$ we need the functional $I_\lambda$ to satisfy the Palais-Smale conditions under a certain level, we will use an argument of \cite[Theorem 3.4]{guedda_veron}.
\begin{lemma}
	\label{lemma:palais-smale}
	$I_\gamma$ satisfies the (PS)$_c$ conditions for all $c<S^{N_\gamma/p} /N_\gamma$.
\end{lemma}
\begin{proof}
	Let $\lbrace u_n\rbrace$ be a sequence in $\mathring{W}^{1,p}_\gamma(\Omega)$ which satisfies the Palais-Smale conditions.
	First of all we claim that
	\begin{equation}
		\label{claim:1}
		\text{the sequence $\lbrace u_n\rbrace$ is bounded in $\mathring{W}^{1,p}_\gamma(\Omega).$}
	\end{equation}
	Indeed, for any $n\in\mathbb{N}$ there exists $k>0$ such that
	\begin{equation}
		\label{eq:1}
		\lvert I_\gamma(u_n) \rvert \le k
	\end{equation}
	and
	\begin{displaymath}
		\left\vert
		\left\langle I'_\gamma (u_n), \tfrac{u_n}{\Vert u_n \Vert_{\gamma,p}} \right\rangle
		\right\vert \leq k
	\end{displaymath}
	and so
	\begin{equation}
		\label{eq:2}
		I_\gamma(u_n) - \frac{1}{p} \braket{I'_\gamma(u_n) , u_n} \le k(1+\norm{u_n}_{\gamma,p}).
	\end{equation}
	Furthermore,
	\begin{displaymath}
		I_\gamma(u_n) - \frac{1}{p} \braket{I'_\gamma(u_n) , u_n} = \left(\frac{1}{p}- \frac{1}{p^*_\gamma}\right) \norm{u_n}_{p_\gamma^*}^{p_\gamma^*} = \frac{1}{N_\gamma} \norm{u_n}_{p_\gamma^*}^{p_\gamma^*}
	\end{displaymath}
	so, thanks to (\ref{eq:2}), we get that for any $n\in\mathbb{N}$
	\begin{equation}
		\label{eq:3}
		\norm{u_n}_{p_\gamma^*}^{p_\gamma^*} \le \hat{k} (1+\norm{u_n}_{\gamma,p})
	\end{equation}
	for a suitable positive constant $\hat{k}.$
	By H\"{o}lder's inequality, we get
	\begin{displaymath}
		\norm{u_n}_{p}^p \le \lvert \Omega \rvert^{{p}/{N_\gamma}} \norm{u_n}_{p_\gamma^*}^p \le \hat{k}^{{p}/{p^*_\gamma}}\lvert \Omega \rvert^{{p}/{N_\gamma}} (1+\norm{u_n}_{\gamma,p})^{p/p^*_\gamma}
	\end{displaymath}
	that is,
	\begin{equation}
		\label{eq:4}
		\norm{u_n}_{p}^p \le \tilde{k} (1+\norm{u_n}_{\gamma,p}),
	\end{equation}
	for a suitable $\tilde{k}>0$ independent of $j$. By (\ref{eq:1}), (\ref{eq:3}) and (\ref{eq:4}), we have that
	\begin{align*}
		k&\ge I_\gamma (u_n) = \frac{1}{p} \norm{u_n}_{\gamma,p}^p - \frac{\lambda}{p} \norm{u_n}_{p}^p - \frac{1}{p^*_\gamma}\norm{u_n}_{p^*_\gamma}^{p^*_\gamma} \\
		&\ge \frac{1}{p} \norm{u_n}_{\gamma,p}^p - \overline{k} (1+\norm{u_n}_{\gamma,p})
	\end{align*}
	for some constant $\overline{k}>0$ independent of $n$, so (\ref{claim:1}) is proved.
 \medskip

	So, there exist a sequence, still denoted by $\lbrace u_n \rbrace$, $u\in \mathring{W}^{1,p}_\gamma(\Omega)$ and $T\in (L^{p'}(\Omega))^N$ such that $ u_n \rightharpoonup u$ weakly in $\mathring{W}^{1,p}_\gamma(\Omega)\cap L^{p^*_\gamma}(\Omega)$ and strongly in $L^q(\Omega), \, 1\le q<p^*_\gamma$ and $\lvert \nabla_\gamma u_n \rvert^{p-2} \nabla_\gamma u_n \rightharpoonup T$ weakly in $ (L^{p'}(\Omega))^N$.
	
	Moreover
	\begin{equation}
		\label{eq:5}
		- \nabla_\gamma\cdot  T - \lvert u \rvert^{p^*_\gamma -2} u - \lambda \lvert u \rvert^{p-2} u = 0
	\end{equation}

	in $W_\gamma^{-1,p'}(\Omega).$
		
	From Theorem \ref{thm:concentration} (Concentration-Compactness) there exist two non-negative bounded measures $\mu,\nu$ in $\Omega$, an at most countable family of points $\lbrace x_j\rbrace_{j\in J}$ and positive numbers $\lbrace \nu_j\rbrace_{j\in J}$ such that
	\begin{displaymath}
	 \lvert \nabla_\gamma u_n\rvert^p\dz \buildrel\ast\over\rightharpoonup \mu, \quad \lvert u_n\rvert^{p^*_\gamma}\dz \buildrel\ast\over\rightharpoonup \nu
	\end{displaymath}
	and
	\begin{align}
		\label{eq:6} &\nu = \lvert u \rvert^{p^*_\gamma}\dz + \sum_{j\in J} \nu_j \delta_{x_j},
		\\
		\label{eq:7} &\mu\ge \lvert \nabla_\gamma u\rvert^p\dz + S\sum_{j\in J}\nu_j^{1-p/N_\gamma} \delta_{x_j}.
	\end{align}
Then, passing to the limit in the expression of $I_\lambda(u_n)$ holds
\begin{equation}
	\label{eq:8}
	c = \frac{1}{p} \int_{\Omega}\, \dmu - \frac{1}{p^*_\gamma} \int_{\Omega}\,\dnu - \frac{\lambda}{p} \int_{\Omega}\lvert u\rvert^p\,\dz.
\end{equation}
Now, given $\varphi\in C^1(\overline{\Omega})$, testing \eqref{eq:5} with the function $u\varphi$
\begin{equation}
	\label{eq:9}
	\int_{\Omega}\left( u T \nabla_\gamma\varphi + \varphi T\nabla_\gamma u\right)\,\dz - \int_{\Omega} \varphi\lvert u\rvert^{p^*_\gamma}\,\dz - \lambda \int_{\Omega} \varphi \lvert u \rvert^p\,\dz = 0,
\end{equation}
moreover, $\braket{I'_\lambda(u_n), \varphi u_n} \to 0 $ as $n\to \infty$, that is
\begin{displaymath}
	\int_{\Omega} u T \nabla_\gamma\varphi\,\dz + \int_{\Omega}\varphi\,\dmu - \int_\Omega\varphi\,\dnu - \lambda \int_{\Omega} \varphi \lvert u\rvert^p\,\dz = 0.
\end{displaymath}
This, combined with \eqref{eq:6} and \eqref{eq:9}, implies
\begin{equation}
	\label{eq:10}
	\int_{\Omega} \varphi\,\dmu = \int_{\Omega}\varphi T \nabla_\gamma u \,\dz + \sum_{j\in J}\nu_j \varphi(x_j).
\end{equation}

Arguing as in the proof of Theorem \ref{thm:concentration}, we pick a function $\varphi\in C^1(\overline{\Omega})$  such that $0\le\varphi\le 1,\, \varphi(0) = 1$ and $\operatorname{supp}\varphi= B(0,1)$. For each~$\varepsilon>0$ we apply  equation \eqref{eq:10} with $\varphi_j = \varphi(\frac{x-x_j}{\varepsilon})$, where $j$ is fixed in $J$. This, combined with \eqref{eq:7} gives
\begin{displaymath}
	\nu_j \ge \int_{\Omega} \varphi_j \left(\lvert \nabla_\gamma u \rvert^p - T \nabla_\gamma u\right)\,\dz + S \nu_j^{1-p/N_\gamma}.
\end{displaymath}
Letting $\varepsilon\to 0$ we get
\begin{displaymath}
	\nu_j \ge S^{N_\gamma/p},\quad\forall j\in J.
\end{displaymath}
As
\begin{displaymath}
	\int_{\Omega} T\nabla_\gamma u\,\dz = \int_{\Omega} \lvert u \rvert^{p^*_\gamma}\,\dz + \lambda \int_{\Omega}\lvert u \rvert^p\,\dz,
\end{displaymath}
combined with \eqref{eq:8} and \eqref{eq:10} considering $\varphi\equiv 1$ gives
\begin{displaymath}
	c = \left(\frac{1}{p} - \frac{1}{p^*_\gamma}\right) \sum_{j\in J} \nu_j + \left(\frac{1}{p} - \frac{1}{p^*_\gamma}\right)\int_{\Omega} \lvert u \rvert^{p^*_\gamma}\,\dz \ge \frac{1}{N_\gamma} S^{N_\gamma/p},
\end{displaymath}
which is in contradiction with the hypothesis $c<S^{N_\gamma}/p/N_\gamma$. We deduce that $J$ is empty and
\begin{displaymath}
\lim_{n	\to\infty}\int_{\Omega}\lvert u_n\rvert^{p^*_\gamma}\,\dz = \int_{\Omega}\lvert u\rvert^{p^*_\gamma}\,\dz
\end{displaymath}
and thus
\begin{equation}
\label{eq:11}
u_n\to u~\text{strongly in}~ L^{p^*_\gamma}(\Omega).
\end{equation}
		
A straightforward computation shows that the sequence $\lbrace A_p(u_n)\rbrace$ is a Cauchy sequence in $W_\gamma^{-1,p'}(\Omega)$. In fact $A_p(u_n) = I'_\lambda(u_n) +\lambda \lvert u_n \rvert^{p-2} u_n + \lvert u_n \rvert^{p^*_\gamma -2} u_n$ and the claim follows from $\eqref{eq:11}$ and the Palais-Smale condition.

Now, a direct application of \cite[Eq (2.2)]{simon} implies
\begin{displaymath}
	\langle A_p(u_n) - A_p(u_m), u_n - u_m 	\rangle \ge
\begin{cases}
	c_1 \norm{u_n - u_m}_{\gamma,p}^p &\hbox{if $p>2$},\\
	c_2 M^{p-2} \norm{u_n-u_m}_{\gamma,p}^2 &\hbox{if $1<p\le 2$},
\end{cases}
\end{displaymath}
where $c_1 = c_1(N,\gamma,p)$, $c_2 = c_2(N,\gamma,p,\Omega)$ and $M = \max\lbrace \norm{u_n}_{\gamma,p}, \norm{u_m}_{\gamma,p}\rbrace$.
So
\begin{displaymath}
	\norm{u_n - u_m}_{\gamma,p}\le
	\begin{cases}
		c_1^{\frac{1}{1-p}} \norm{A_p(u_n) - A_p(u_m)}_{W^{-1,p'}_\gamma(\Omega)}^{\frac{1}{p-1}}  &\hbox{if $p>2$},\\
		c_2^{-1} M^{2-p} \norm{A_p(u_n) - A_p(u_m)}_{W^{-1,p'}_\gamma(\Omega)} &\hbox{if $1<p\le 2$},
	\end{cases}
\end{displaymath}
so we deduce that $u_n \to u$ strongly in $\mathring{W}^{1,p}_\gamma(\Omega)$.
\end{proof}

	If $\lambda_{k+m}<\lambda_{k+m+1}$ then $i(\Psi^{\lambda_{k+m}}) = k+m$ by \eqref{eq:index}. We now construct a symmetric subset $A_0$ of $\Psi^{\lambda_{k+m}}$ with the same cohomological index.
 We need some further properties of the operator $A_p$ introduced in section \ref{sec:eigenvalues}.
	\begin{lemma}
		The operator $A_p$ is strictly monotone, i.e.,
		\begin{displaymath}
			\langle A_p(u) - A_p(v), u-v\rangle >0
		\end{displaymath}
		for all $u\ne v$ in $\mathring{W}^{1,p}_\gamma(\Omega).$
	\end{lemma}

	\begin{proof}
		In order to apply \cite[Lemma 6.3]{perera_p_laplacian}, it suffices to show that	
		\begin{displaymath}
			\langle A_p(u),v\rangle \le \norm{u}_{\gamma,p}^{p-1}\norm{v}_{\gamma,p}
		\end{displaymath}
		for every $u,v\in \mathring{W}^{1,p}_\gamma(\Omega)$ and the equality holds if and only if $\alpha u = \beta v$ for some $\alpha \geq 0$, $\beta\ge 0$, not both zero.
		
		By H\"{o}lder's inequality,
		\begin{displaymath}
		\langle A_p(u),v\rangle = \int_{\Omega} \lvert \nabla_\gamma u \rvert^{p-2}\nabla_\gamma u \nabla_\gamma v\,\dz \le \int_{\Omega} \lvert \nabla_\gamma u \rvert^{p-1} \lvert \nabla_\gamma v \rvert \le \norm{u}_{\gamma,p}^{p-1}\norm{v}_{\gamma,p}.
		\end{displaymath}
		
	Clearly, equality holds throughout if and only if $\alpha u = \beta v$ for some $\alpha \geq 0$, $\beta\ge 0$, not both zero.
	
	Conversely, if $\langle A_p(u),v\rangle =\norm{u}_{\gamma,p}^{p-1}\norm{v}_{\gamma,p}$, equality holds in both inequalities. The equality in the H\"{o}lder's inequality gives
	\begin{displaymath}
	\alpha \lvert \nabla_\gamma u \rvert = \beta \lvert \nabla_\gamma v \rvert\quad\hbox{a.e.}
	\end{displaymath}
	for some $\alpha,\beta\ge 0$, not both zero, and the equality in the Schwartz inequality gives
	\begin{displaymath}
	\alpha \nabla_\gamma u = \beta \nabla_\gamma v\quad\hbox{a.e.,}
	\end{displaymath}
	so $\alpha u = \beta v$.
	\end{proof}

\begin{lemma}
	\label{lemma:1}
For each $w\in L^p(\Omega),$ the problem
\begin{displaymath}
\begin{cases}
-\Delta_\gamma^p u = \lvert w \rvert ^{p-2} w\quad&\hbox{in $\Omega$},\\
u=0\quad&\hbox{on $\partial\Omega$},
\end{cases}
\end{displaymath}
has a unique weak solution $u\in \mathring{W}^{1,p}_\gamma(\Omega).$ Moreover, the map $J\colon L^p(\Omega)\to \mathring{W}^{1,p}_\gamma(\Omega)$, $w\mapsto u$ is continuous.
\end{lemma}
\begin{proof}
The existence of a solution follows from a straightforward minimization procedure, and uniqueness is immediate from the strict monotonicity of the operator $A_p$. Let $w_j\to w$ in $L^p(\Omega)$ and let $u_j = J(w_j)$, $u=J(w)$. By definition,
\begin{displaymath}
	\langle A_p(u_j), v \rangle = \int_\Omega \lvert w_j \rvert^{p-2}w_j v\,\dz \quad\hbox{for all $v\in \mathring{W}^{1,p}_\gamma(\Omega)$}.
\end{displaymath}
Testing with $v=u_j$ gives
\begin{displaymath}
\norm{u_j}_{\gamma,p}^p \le \norm{w_j}_p^{p-1} \norm{u_j}_p^p
\end{displaymath}
by H\"{o}lder's inequality, which together with the continuity of the embedding $\mathring{W}^{1,p}_\gamma(\Omega)\hookrightarrow L^p(\Omega)$ shows that $\lbrace u_j\rbrace$ is bounded in $\mathring{W}^{1,p}_\gamma(\Omega)$. By reflexivity, up to a subsequence $u_j \rightharpoonup u^*$ in $\mathring{W}^{1,p}_\gamma(\Omega)$ and strongly in $L^p(\Omega)$.

By the continuity of the Nemitskii operator,
\begin{displaymath}
	\lvert w_j\rvert ^{p-2} w_j \to \lvert w\rvert ^{p-2}w
\end{displaymath}
strongly in $L^{p'}(\Omega)$. Therefore $A_p (u_j) \to A_p(u)$, or
\begin{displaymath}
	\langle A_p(u_j), v \rangle = \int_\Omega \vert w_j \vert^{p-2} w_j v \, \dz \to \int_\Omega \vert w \vert^{p-2} w v \, \dz = \langle A_p(u),v \rangle
\end{displaymath}
for every $v \in \mathring{W}^{1,p}_\gamma(\Omega)$. Hence $J(w)=u$. Now,
\begin{displaymath}
	\langle A_p(u_j),u_j-u^* \rangle = \int_\Omega \vert w_j \vert^{p-2} w_j (u_j-u^*) \, \dz \to 0
\end{displaymath}
since $\lbrace \vert w_j\vert^{p-2} w_j \rbrace$ is bounded in $L^{p'}(\Omega)$ and $u_j \to u^*$ strongly in $L^p(\Omega)$. Since $A_p$ is of type (S), up to another subsequence $u_j \to u^*$ strongly in $\mathring{W}^{1,p}_\gamma(\Omega)$. It follows easily that $J(u^*)=w$, and by uniqueness $u^*=u$. A standard argument shows now that the whole sequence $\lbrace u_j \rbrace$ converges to $u$, since each subsequence converges to the same limit $u$.
\end{proof}

\begin{prop}
	\label{prop:compact}
If $\lambda_\ell < \lambda_{\ell+1}$, then $\Psi^{\lambda_\ell}$ has a compact symmetric subset $A_0$ with $i(A_0) = \ell$.
\end{prop}

\begin{proof}
	Let
	\begin{displaymath}
	\pi_p(u) = \frac{u}{\norm{u}_p}, \quad u\in \mathring{W}^{1,p}_\gamma(\Omega)\setminus\lbrace 0\rbrace,
	\end{displaymath}
	be the radial projection onto $\mathcal{M}_p = \lbrace u\in \mathring{W}^{1,p}_\gamma(\Omega)\mid \norm{u}_p = 1\rbrace$, and let
	\begin{displaymath}
	A = \pi_p(\Psi^{\lambda_\ell}) = \lbrace w\in\mathcal{M}_p\mid \norm{u}_{\gamma,p}^p\le \lambda_\ell\rbrace,
	\end{displaymath}
	which is compact in $L^p(\Omega)$ since the embedding $\mathring{W}^{1,p}_\gamma(\Omega)\hookrightarrow L^p(\Omega)$ is compact.

	Then $i(A) = i(\Psi^{\lambda_\ell}) = \ell$ by (i$_2$) of Proposition \ref{prop:index} and \eqref{eq:index}. For $w\in A$, let $u = J(w)$, where $J$ is the map defined in Lemma \ref{lemma:1}, so
	\begin{displaymath}
	\langle A_p(u),v\rangle = \int_{\Omega} \lvert w \rvert^{p-2} w v\,\dz, \quad\forall v\in \mathring{W}^{1,p}_\gamma(\Omega).
	\end{displaymath}
	Testing with $v=u,w$ and using H\"{o}lder's inequality gives
	\begin{displaymath}
	\norm{u}_{\gamma,p}\le \norm{w}_p^{p-1}\norm{u}_p = \norm{u}_p, \quad 1 = \langle A_p(u), w\rangle \le \norm{u}_{\gamma,p}^{p-1}\norm{w}_{\gamma,p},
	\end{displaymath}
	so
	\begin{displaymath}
	\norm{\pi_p(u)}_{\gamma,p} = \frac{\norm{u}_{\gamma,p}}{\norm{u}_p} \le \norm{w}_{\gamma,p}
	\end{displaymath}
	and hence $\pi_p(u)\in A$.
	
	Let $\widetilde{J} = \pi_p\circ J$ and let $\widetilde{A} = \widetilde{J}(A)\subset A$. Since $\widetilde{J}$ is an odd continuous map from $L^p(\Omega)$ to $\mathring{W}^{1,p}_\gamma(\Omega)$ and $A$ is compact in $L^p(\Omega)$, then $\widetilde{A}$ is a compact set and $i(\widetilde{A}) = i(A) = \ell$ by (i$_2$) of Proposition \ref{prop:index}. Let
	\begin{displaymath}
	\pi(u) = \frac{u}{\norm{u}_{\gamma,p}}, \quad u\in \mathring{W}^{1,p}_\gamma(\Omega)\setminus\lbrace 0\rbrace,
	\end{displaymath}
	be the radial projection onto $\mathcal{M} = \lbrace u\in \mathring{W}^{1,p}_\gamma(\Omega)\mid \norm{u}_{\gamma,p} = 1\rbrace$ and let $A_0 = \pi(\widetilde{A})$. Then $A_0\subset \Psi^{\lambda_\ell}$ is compact and $i(A_0) = i(\widetilde{A}) = \ell$.
\end{proof}

Now we are ready to prove Theorem \ref{thm:main}.
\begin{proof}[Proof of Theorem \ref{thm:main}]
	 We only give the proof of (2). The proof of (1) is similar and simpler. By Lemma \ref{lemma:palais-smale}, $I_\lambda$ satisfies the (PS)$_c$ conditions for all $c<\frac{S^{N_\gamma/p}}{N_\gamma}$, so we apply Theorem \ref{thm:critical_point} with $b = \frac{S^{N_\gamma/p}}{N_\gamma}$. By Proposition \ref{prop:compact}, $\Psi^{\lambda_{k+m}}$ has a compact symmetric subset $A_0$ with
	 \begin{displaymath}
	 i(A_0) = k+m.
	 \end{displaymath}
	 We take $B_0 = \Psi_{\lambda_{k+1}}$, so that
	 \begin{displaymath}
	 i(S_1\setminus B_0) = k
	 \end{displaymath}
	 by \eqref{eq:index}. Let $R>r>0$ and let $A,B$ and $X$ be as in Theorem \ref{thm:critical_point}. For $u\in B_0$,
	 \begin{displaymath}
	 I_\lambda(ru) \ge \frac{r^p}{p}\left(1 - \frac{\lambda}{\lambda_{k+1}}\right) - \frac{r^{p^*_\gamma}}{p^*_\gamma S^{p^*_\gamma/p}}
	 \end{displaymath}
	 by the definition \eqref{eq:sobconst} of the constant $S$. Since $\lambda<\lambda_{k+1}$ and $p^*_\gamma>p$, it follows that $\inf I_\lambda(B)>0$ if $r$ sufficiently small.
	 For $u\in A_0\subset \Psi^{\lambda_{k+1}}$,
	 \begin{displaymath}
	 I_\lambda(Ru) \le \frac{R^p}{p} \left(1- \frac{\lambda}{\lambda_{k+1}}\right) - \frac{R^{p^*_\gamma}}{p^*_\gamma \lvert \Omega \rvert^{p^*_\gamma/N} \lambda_{k+1}^{p^*_\gamma/p}}
	 \end{displaymath}
	 by H\"{o}lder's inequality, so there exists $R>r$ such that $I_\lambda\le 0$ on $A$. For $u\in X$,
	\begin{align*}
		I_\lambda(u)&\le \frac{\lambda_{k+1} - \lambda}{p} \int_{\Omega} \lvert u\rvert^p\,\dz - \frac{1}{p^*_\gamma \lvert \Omega \rvert^{p^*_\gamma/N_\gamma}}\left(\int_{\Omega}\lvert u \rvert^p\,\dz \right)^{p^*_\gamma/p}\\
		&\le \sup_{\rho\ge 0} \left[\frac{(\lambda_{k+1} - \lambda)\rho}{p} - \frac{\rho^{p^*_\gamma/p}}{p^*_\gamma \lvert \Omega \rvert^{p^*_\gamma/N_\gamma}}\right] = \frac{\lvert \Omega \rvert}{N_\gamma} (\lambda_{k+1} -\lambda)^{N_\gamma/p}.
	\end{align*}
	So,
	\begin{displaymath}
	\sup I_\lambda(X) \le \frac{\lvert \Omega \rvert}{N_\gamma} (\lambda_{k+1} - \lambda)^{N_\gamma/p} < \frac{S^{N_\gamma/p}}{N_\gamma}
	\end{displaymath}
	by \eqref{eq:lambdak}. Theorem \ref{thm:critical_point} now gives $m$ distinct pairs of (nontrivial) critical points $\pm u_j^\lambda, \,j=1,\dots, m,$ of $I_\lambda$ such that
	\begin{displaymath}
	0<I_\lambda(u_j^\lambda)\le \frac{\lvert \Omega \rvert}{N_\gamma} (\lambda_{k+1} - \lambda)^{N_\gamma/p} \to 0 ~\hbox{as $\lambda \nearrow \lambda_{k+1}$.}
	\end{displaymath}
	Then
	\begin{displaymath}
		\norm{u_j^\lambda}_{p^*_\gamma}^{p^*_\gamma} = N_\gamma \left[I_\lambda(u_j^\lambda) - \frac{1}{p} \braket{I'_\lambda(u_j^\lambda), u_j^\lambda}\right] = N_\gamma I_\lambda (u_j^\lambda) \to 0
	\end{displaymath}
	as $\lambda\nearrow \lambda_{k+1},$ and hence $u_j^\lambda\to 0$ in $L^p(\Omega)$ also by H\"{o}lder's inequality, so
	\begin{displaymath}
		\norm{u_j^\lambda}_{\gamma,p}^p = p I_\lambda(u_j^\lambda) + \lambda \norm{u_j^\lambda}_p^p + \frac{p}{p^*_\gamma} \norm{u_j^\lambda}_{p^*_\gamma}^{p^*_\gamma} \to 0,
	\end{displaymath}
	as $\lambda\nearrow\lambda_{k+1},$ this concludes the proof.	
\end{proof}

\bibliography{p_Grushin_multiplicity}{}
\bibliographystyle{plain}

\bigskip

Paolo Malanchini, Dipartimento di Matematica e Applicazioni, Universit\`a degli Studi di Milano Bicocca, via R. Cozzi 55, I-20125 Milano, Italy.

\medskip

Giovanni Molica Bisci, Department of Human Sciences and Promotion of Quality of Life,
San Raffaele University,
via di Val Cannuta 247,
I-00166 Roma, Italy.

\medskip

Simone Secchi, Dipartimento di Matematica e Applicazioni, Universit\`a degli Studi di Milano Bicocca, via R. Cozzi 55, I-20125 Milano, Italy.
\end{document}